\documentclass[12pt,a4paper]{amsart}
\usepackage{amssymb,color}
\usepackage[english]{babel}          
\usepackage{verbatim}
\usepackage[T1]{fontenc}

\usepackage{floatflt,graphicx}
\usepackage{color,xcolor}
\usepackage{enumerate}
\usepackage{animate}
\usepackage{amsmath}
\usepackage{amsthm}

\newcounter{minutes}
\divide\time by 60
\newcounter{hours}
\multiply\time by 60 \addtocounter{minutes}{-\time}
\title[{}]{Some remarks on the Cassinian metric}

{\small

\author{Parisa Hariri}
\address{Department of Mathematics and Statistics,
  University of Turku, Turku, Finland}
\curraddr{}
\email{parisa.hariri@utu.fi}

\author{Riku Kl\'en}
\address{Department of Mathematics and Statistics,
  University of Turku, Turku, Finland}
\curraddr{}
\email{riku.klen@utu.fi}

\author{Matti Vuorinen}
\address{Department of Mathematics and Statistics,
  University of Turku, Turku, Finland}
\email{vuorinen@utu.fi}

\author{Xiaohui Zhang}
\address{Department of Physics and  Mathematics, University of Eastern Finland, 80101 Joensuu, Finland}
\email{xiaohui.zhang@uef.fi}

}

\keywords{hyperbolic metric, triangular ratio metric, Cassinian metric, distance ratio metric}

\subjclass[2010]{51M10, 30C65}

\date{}

\dedicatory{}

\commby{}

\theoremstyle{plain}
\newtheorem{thm}[equation]{Theorem}
\newtheorem{cor}[equation]{Corollary}
\newtheorem{lem}[equation]{Lemma}

\theoremstyle{definition}

\theoremstyle{remark}
\newtheorem{rem}[equation]{Remark}

\newtheorem{nonsec}[equation]{}

\numberwithin{equation}{section}

\newcommand{\beq}{\begin{equation}}
\newcommand{\eeq}{\end{equation}}

\newcommand{\ben}{\begin{enumerate}}
\newcommand{\een}{\end{enumerate}}

\newcommand{\bequu}{\begin{eqnarray*}}
\newcommand{\eequu}{\end{eqnarray*}}

\newcommand{\bequ}{\begin{eqnarray}}
\newcommand{\eequ}{\end{eqnarray}}

\newcommand{\B}{\mathbb{B}^2}

\newcommand{\Bn}{ {\mathbb{B}^n} }
\newcommand{\Rn}{ {\mathbb{R}^n} }

\newcommand{\arth}{\,\textnormal{arth}}

\newcommand{\sh}{\,\textnormal{sh}}

\renewcommand{\th}{\,\textnormal{th}}

\renewcommand{\Im}{{ \rm Im}\,}
\renewcommand{\Re}{{ \rm Re}\,}

\begin{document}
\def\thefootnote{}
\footnotetext{ \texttt{\tiny File:~\jobname .tex,
           printed: \number\year-\number\month-\number\day,
           \thehours.\ifnum\theminutes<10{0}\fi\theminutes}
} \makeatletter\def\thefootnote{\@arabic\c@footnote}\makeatother

\begin{abstract}
Some sharp inequalities between the triangular ratio metric and the Cassinian metric are proved in the unit ball.
\end{abstract}

\maketitle


\section{Introduction}\label{section1}

\setcounter{equation}{0}

Geometric function theory makes use of metrics in many ways. In the distortion theory,
which is a significant part of function theory, one seeks to estimate the distance between $f(z)$  and $f(w)$
for a given analytic function $f$ of the unit disk $  \mathbb{B}^2$ in terms of the distance between $z$ and $w$
 and their position in  $  \mathbb{B}^2$\, \cite{b,kl}. Distances are often measured in terms of
 hyperbolic or, in the case of multidimensional theory, hyperbolic type metrics, see \cite{gp,himps,k}.
Some examples of recurrent metrics are the quasihyperbolic, distance ratio, and Apollonian metrics, see \cite{gp,b2,himps}. In this paper we shall study a metric recently introduced by
Z. Ibragimov \cite{i}, the Cassinian metric, and relate it to some of these other metrics. For this
purpose we first recall the definitions of the hyperbolic metric of the unit ball ${\mathbb B}^n$
in ${\mathbb R}^n$ and the distance ratio metric of a proper open subset of  ${\mathbb R}^n$.


\begin{nonsec}{\bf Hyperbolic metric. }
Recall the definition of the hyperbolic distance $\rho_{\mathbb{B}^n}(x,y)$
between two points $x,y \in \mathbb{B}^n $ \cite{b}:
\begin{eqnarray}\label{tro}
\th{\frac{\rho_{\mathbb{B}^n}(x,y)}{2}}&=&\frac{|x-y|}{\sqrt{|x-y|^2+(1-|x|^2)(1-|y|^2)}}.
\end{eqnarray}
Here $\th$ stands for the hyperbolic tangent function.

\end{nonsec}

\begin{nonsec}{\bf Distance ratio metric.}
Let $G \subset {\mathbb R}^n\,$ be a proper open subset of ${\mathbb R}^n$ and for $x \in G$ let
$d(x,\partial G)= \inf \{  |x-y|:  y \in \partial G\}\,.$ Then
for all
$x,y\in G$, the  distance ratio
metric $j_G$ is defined as
\begin{eqnarray*}
 j_G(x,y)=\log \left( 1+\frac{|x-y|}{\min \{d(x,\partial G),d(y, \partial G) \} } \right)\,.
\end{eqnarray*}
This metric was introduced in
\cite{gp,go} in a slightly different form and in the above form in \cite{vu0}.
It is a standard tool in the study of metrics, see e.g. \cite{chkv,himps, imsz,k}. If confusion seems unlikely, then we also write $d(x)= d(x,\partial G)\,.$
\end{nonsec}

\begin{lem}\cite[  Lemma 2.41(2)]{vu}, \cite[ Lemma 7.56]{avv}\label{10} Let $G \in \{  {\mathbb B}^n , {\mathbb H}^n\}\,,$ and let
 $\rho_G$ stand for the respective hyperbolic metric. Then for all $x,y\in G$
 $$ j_G(x,y)\le \rho_G(x,y) \le 2j_G(x,y).$$
\end{lem}

In his paper \cite{h}, P. H\"ast\"o studied a general family of metrics. A particular case is the
Cassinian metric defined as follows for a domain $G \subsetneq   \mathbb{R}^n$ and
$x,y \in G$:
\begin{equation}\label{cm}
c_G(x,y)=\sup_{z\in \partial G}\frac{|x-y|}{|x-z||z-y|}\,.
\end{equation}
The term "Cassinian metric" was introduced by Z. Ibragimov  in \cite{i}, and the geometry of the Cassinian metric including geodesics, isometries, and completeness was first studied there.
 Another, similar metric is the triangular ratio metric, which we studied in \cite{chkv}. It is
defined as follows for a domain $G \subsetneq   \mathbb{R}^n$ and $x,y \in G$:
\begin{equation}\label{sm}
s_G(x,y)=\sup_{z\in \partial G}\frac{|x-y|}{|x-z|+|z-y|}\in [0,1]\,.
\end{equation}
The triangular ratio metric is also a particular case of
the metrics considered in  \cite{h}. For the case $G= \mathbb{B}^n$ this metric is closely related to the
hyperbolic metric as the following theorem shows.

\begin{thm} (\cite[2.17]{hvz})\label{hvz216} For $x,y \in \mathbb{B}^n$
$$
\th{\frac{\rho_{\mathbb{B}^n}(x,y)}{4}} \le s_{\mathbb{B}^n} (x,y) \le \th{\frac{\rho_{\mathbb{B}^n}(x,y)}{2}} \,.
$$
\end{thm}

We have been unable to find an explicit formula
for $s_{\mathbb{B}^2}(x,y)$. In the special case $|x|=|y|$ we will give such a formula in Theorem \ref{sb02}.

Very recently, the Cassinian metric and its relation to other metrics,
in particular, to the hyperbolic metric, were discussed by Ibragimov, Mohapatra, Sahoo, and Zhang
in \cite{imsz}. Also geometric properties of the Cassinian metric have
been studied in \cite{kms}. One of the main results of \cite{imsz} is the following theorem.

\begin{thm} (\cite[3.1]{imsz})\label{imsz31}
 For $x,y \in \mathbb{B}^n$
$$\sh{\frac{\rho_{\mathbb{B}^n}(x,y)}{2}} \le c_{\mathbb{B}^n} (x,y)\,.
$$
Moreover, here equality holds for $x=-y\,.$
\end{thm}

The equality statement was not pointed out in \cite[3.1]{imsz}, but it follows from
$$   c_{\mathbb{B}^n}(x, -x) = \frac{2|x|}{1-|x|^2}\,$$
because by  \eqref{tro} for $x,y \in \mathbb{B}^n$
$$
\sh{\frac{\rho_{\mathbb{B}^n}(x,y)}{2}}=\frac{|x-y|}{\sqrt{(1-|x|^2)(1-|y|^2)}}.
$$

Hence Theorem \ref{imsz31} is sharp. However, we will refine it in Section 2, see Remark \ref{betterrmk}.

Our goal here is to continue this study.
A part of this process is to compare the Cassinian metric to several other
widely known metrics such as the triangle ratio metric and the distance ratio metric
of the unit ball $\mathbb{B}^n  \,.$ The main result is the following sharp theorem.

\begin{thm}\label{2sc}
 Suppose that $D$ is a subdomain of $\mathbb{B}^n\,.$ Then for $x, y\in D$ we have
 \[
   2 s_{D}(x,y)\leq c_{D}(x,y).
 \]
In the case $D=\mathbb{B}^n$, the constant 2 on the left-hand side is best possible.
\end{thm}

Let us next compare this result to Theorem \ref{imsz31}. The identity
$$
\sh (\arth t) = \frac{t}{\sqrt{1-t^2}}\,, t > 0\,,
$$
together with Theorem \ref{hvz216} implies for $x,y \in \mathbb{B}^n $
\begin{equation} \label{srho}
 \frac{s_{\mathbb{B}^n}(x,y)}{\sqrt{1-s_{\mathbb{B}^n}(x,y)^2}} = \sh( \arth s_{\mathbb{B}^n}(x,y)) \le \sh{\frac{\rho_{\mathbb{B}^n}(x,y)}{2}} \,.
\end{equation}
In combination with \eqref{srho} Theorem \ref{imsz31} yields for $x,y \in \mathbb{B}^n $
\begin{equation} \label{sandc}
s_{\mathbb{B}^n}(x,y) \le \frac{c_{\mathbb{B}^n}(x,y)}{\sqrt{1+c_{\mathbb{B}^n}(x,y)^2}} \,.
\end{equation}
We see that Theorem \ref{2sc} gives a better bound than \eqref{sandc} for $c_{\mathbb{B}^n}(x,y) < \sqrt{3} \,.$


Finally, we study the growth of the Cassinian metric under quasiregular mappings of the unit
disk onto itself.

\begin{thm}\label{casgrow}
If $f:\B\rightarrow {\mathbb R}^2$ is a non-constant $K$-quasiregular map with $f\B\subset\B$ and $f(0)=0$, then
$$c_{\B}(0,f(x))\le e^{\pi(K-1/K)} \max \{c_{\B}(0,x)^{1/K},  c_{\B}(0,x) \}$$
for all $x\in\B\,.$
\end{thm}

\section{Preliminary results}\label{section2}


In this section we will prove some sharp inequalities between the Cassinian metric and the distance ratio metric. For this purpose we need the following technical lemma.

\begin{lem}\label{loglemma}
  (1) The function $f(x) = x^{-1} \log(1+x)$ is decreasing on $(0,\infty)$.

  (2) Let $a>0$. The function
  \[
    g(x) = \frac{\log ax}{a-\frac{1}{x}}
  \]
  is increasing on $(0,\infty)$.

  (3) The function
  \[
    h(x) = \frac{\log \frac{1+x}{1-x}}{\frac{1}{1-x}-\frac{1}{1+x}}
  \]
  is decreasing on $(0,1)$.

  (4) Let $x\in (0,1)$. The function
\[
f(b)=\frac{\log\left(1+\frac{b}{1-x}\right)}{\log\left(1+\frac{b}{(1-x)\,(b+1-x)}\right)},
\]
is increasing on $(0,2)$.
\end{lem}

\begin{proof}
  (1) By \cite[4.1.33]{as}, we easily obtain that for all $x>0$
  \[
    f'(x) = \frac{\frac{x}{1+x}-\log(1+x)}{x^2} <0\,.
  \]

  (2) Now for all $a,x > 0$ by \cite[4.1.33]{as}
  \[
    g'(x) = \frac{ax-(1+\log(ax))}{(1-ax)^2} >0 \,.
  \]

  (3) Recall first that $\log(1+x)>\frac{2x}{2+x}$, for $x>0$. Using this inequality we see that
  \[
    h'(x) = \frac{2x-(1+x^2)\log\frac{1+x}{1-x}}{2x^2}<0.
  \]

  (4) We have
\[
f'(b)=\frac{(1 - x) \log\left(1+\frac{b}{1-x}\right) + (b (x-2) - (x-1)^2) C}{(b (x-2) - (x-1)^2) (1 + b -
   x) C^2}:=\frac{A}{B},
\]
where
\[
  C = \log\left(
   1 + \frac{b}{(b + (x-1)^2 - b x)}\right).
\]
Because $x \in (0,1)$ we see that $B<0$ and therefore it is enough to show that $A<0$. Now
\[
A'(b)=(x-2)\log\left(
   1 + \frac{b}{(b + (x-1)^2 - b x)}\right),
\]
and
\[
A''(b)=\frac{(x-2) (x-1)}{(b (x-2) - (x-1)^2) (1 + b - x)}
\]
which is negative, therefore $A'(b)$ is decreasing and $A'(b)<A'(0)=0$. Hence $A(b)$  is decreasing and $A(b)<A(0)=0$.
\end{proof}


For a domain $G \subsetneq \Rn$ we define the quantity
\[
  \hat{c}_G(x,y) = \frac{|x-y|}{|x-z||z-y|},
\]
where $x,y \in G \subsetneq \Rn$ and
\[
  \begin{array}{ll} z \in \partial G \cap  S^{n-1}(x,d(x)) \textrm{ such that } |z-y| \textrm{ is minimal}, & \textrm{if } d(x) \le d(y),\\
  z \in \partial G \cap  S^{n-1}(y,d(y)) \textrm{ such that } |z-x| \textrm{ is minimal}, & \textrm{if } d(y) < d(x). \end{array}
\]
Clearly for all domains $G$ and for all points $x,y \in G$ we have $\hat{c}_G(x,y) \le c_G(x,y)$.

\begin{thm}\label{jandc}
For all $x, y \in \mathbb{B}^n $ we have
\[
j_{\mathbb{B}^n} (x, y) \le
a \, \log(1+c_{\mathbb{B}^n}(x, y))\,,
\]
where
\[
  a=\frac{\log\left(\frac{1+\alpha}{1-\alpha}\right)}{\log\left(\frac{1+2\alpha-\alpha^2}{(1-\alpha^2)}\right)} \approx 1.3152
\]
and $\alpha \in (0,1)$ is the solution of the equation
\[
  (1+t^2)\log \frac{1+t}{1-t}+(t^2-2t-1)\log \frac{1+2t-t^2}{1-t^2}=0.
\]
\end{thm}
\begin{proof}
By the definition of $\hat{c}_{\mathbb{B}^n}(x,y)$, it is enough to show that
\[
j_{\mathbb{B}^n} (x, y) \le
a \, \log(1+\hat{c}_{\mathbb{B}^n}(x, y))\,.
\]
Assume $|y|\leq |x|$, and denote $y'=|x|-|x-y|\,.$ Then by geometry
\bequu
\frac{\log\left(1+\frac{|x-y|}{1-|x|}\right)}{\log\left(1+\frac{|x-y|}{(1-|x|)\,|y-e_1|}\right)} &\leq & \frac{\log\left(1+\frac{|x-y|}{1-|x|}\right)}{\log\left(1+\frac{|x-y|}{(1-|x|)\,|y'-e_1|}\right)}\\
&=& \frac{\log\left(1+\frac{|x-y|}{1-|x|}\right)}{\log\left(1+\frac{|x-y|}{(1-|x|)\,(|x-y|+1-|x|)}\right)}\,.
\eequu
If we denote $b=|x-y|$, then by \ref{loglemma} (4),
\[
f(b)=\frac{\log\left(1+\frac{b}{1-|x|}\right)}{\log\left(1+\frac{b}{(1-|x|)\,(b+1-|x|)}\right)}
\]
is increasing.

Thus
\[
f(b)\leq f(2|x|)=\frac{\log\left(\frac{1+|x|}{1-|x|}\right)}{\log\left(\frac{1+2|x|-|x|^2}{1-|x|^2}\right)}:=m(|x|).
\]
The function $m(t)$ attains its maximum when
\[
  (1+t^2)\log \frac{1+t}{1-t}+(t^2-2t-1)\log \frac{1+2t-t^2}{1-t^2}=0\,,
\]
and by numerical computation we see that $m(|x|)$ has its maximal value $m(\alpha)\approx 1.3152= a$  when $|x|= \alpha \approx 0.6564$.\qedhere
\end{proof}

\begin{rem} \label{betterrmk} If we combine the results \ref{10}, \ref{hvz216}, and \ref{2sc} we get that
for all $x,y \in \mathbb{B}^n$ we have
$$ j_{\mathbb{B}^n} (x,y)\le 4 \arth (c_{\mathbb{B}^n}(x,y)/2). $$
Let  $a$ be as in Theorem \ref{jandc}.  It is easy to check that for all $t >0$
$$
a \log(1+t)\leq4\arth(t/2),
$$
which implies that Theorem \ref{jandc} gives a better estimate than what we can get from the
results in Section 1.  
\end{rem}

The next two results refine \cite[Corollary 3.5]{imsz} and give a sharp constant.

\begin{thm}\label{jandctildeinBn}
For all $x, y \in \mathbb{B}^n $ we have
\[
  j_{\mathbb{B}^n} (x, y) \le \hat{c}_{\mathbb{B}^n}(x, y)\,.
\]
Moreover, the right hand side cannot be replaced with $\lambda \hat{c}_{\mathbb{B}^n}(x,y)$ for any $\lambda\in (0,1)$\,.
\end{thm}
\begin{proof}
  We denote $G=\mathbb{B}^n$ and may assume $d(x) \le d(y)$ and $x \neq y\,.$ We first fix $|x|$. Now by writing $t = |x-y|/(1-|x|) > 0$ we obtain
  \[
    \frac{j_G(x,y)}{\hat{c}_G(x,y)} = \frac{\log \left( 1 + \frac{|x-y|}{1-|x|} \right) }{\frac{|x-y|}{(1-|x|) \left| y-\frac{x}{|x|} \right| }} = \frac{\log (1+t)}{t} \left| y-\frac{x}{|x|} \right|.
  \]
  Next we fix $|y-x/|x||$ and by Lemma \ref{loglemma} (1) and the triangle inequality it is clear that $|x-y| \ge |y-x/|x||-(1-|x|)$. We denote $s = |y-x/|x|| \in (1-|x|,1+|x|]$ and obtain
  \[
    \frac{j_G(x,y)}{\hat{c}_G(x,y)} = \frac{\log (1+t)}{t}s \le \frac{\log \left( 1+\frac{s-(1-|x|)}{1-|x|} \right)}{\frac{s-(1-|x|)}{1-|x|}} s = \frac{\log \frac{s}{1-|x|} }{\frac{1}{1-|x|}-\frac{1}{s}}.
  \]
Since $s \le 1+|x|$ we have by Lemma \ref{loglemma} (2)
  \[
    \frac{j_G(x,y)}{\hat{c}_G(x,y)} \le \frac{\log \frac{s}{1-|x|} }{\frac{1}{1-|x|}-\frac{1}{s}} \le \frac{\log \frac{1+|x|}{1-|x|} }{\frac{1}{1-|x|}-\frac{1}{1+|x|}}.
  \]
  Using these results we find an upper bound for this quantity in terms of $|x|$ and obtain by Lemma \ref{loglemma} (3)
  \[
    \frac{j_G(x,y)}{\hat{c}_G(x,y)} \le \frac{\log \frac{1+|x|}{1-|x|} }{\frac{1}{1-|x|}-\frac{1}{1+|x|}} \le \lim_{|x| \to 0} \frac{\log \frac{1+|x|}{1-|x|} }{\frac{1}{1-|x|}-\frac{1}{1+|x|}} =1,
  \]
  and the assertion follows.

  Finally, suppose that $\lambda \in (0,1)\,$ and  $j_\Bn(x,y) \le \lambda \hat{c}_\Bn(x,y)$ for all $x,y \in \Bn$. This yields
  $$ j_\Bn(x,0)=\log \left( 1+ \frac{|x|}{1-|x|} \right) \le \lambda \hat{c}_\Bn(x,0) =\lambda   \frac{|x|}{1-|x|}\,.$$
 Letting $|x|\to 0$ yields a contradiction.
\end{proof}


\begin{cor} \label{jand2c}
For all $x, y \in \mathbb{B}^n $ we have
\[
  j_{\mathbb{B}^n} (x, y) \le c_{\mathbb{B}^n}(x, y)\,.
\]
Moreover, the right hand side cannot be replaced with $\lambda {c}_{\mathbb{B}^n}(x,y)$ for any $\lambda\in (0,1)$\,.
\end{cor}
%


\section{A formula for triangular ratio metric}\label{section3}


It seems to be a challenging problem to find an explicit formula for $s_{\Bn}(x,y)$ for given $x, y\in \Bn$. We shall give in this section a
simple formula for $s_{\B}(a,b)$ in the case when $|a|=|b|<1$.

\begin{thm}\label{sb02}
Let $a=\alpha+i \beta$, $\alpha, \beta >0$, be a point in the unit disk. If $\left|a-1/2\right|>1/2$, then $s_{\B}(a,\bar{a})=|a|$ and otherwise
\beq\label{sbeta}
s_{\B}(a, \bar{a})=\frac{\beta}{\sqrt{(1-\alpha)^2+\beta^2}} \le |a|=\sqrt{\alpha^2+\beta^2}.
\eeq
\end{thm}

\begin{proof}
From the definition of the triangular ratio metric it follows that
\[
s_{\B}(a, \bar{a})=\frac{|a-\bar{a}|}{|a-z|+|\bar{a}-z|}=\frac{2\Im (a)}{|a-z|+|\bar{a}-z|}
\]
for some point $z=u+i v$. In order to find $z$ we consider the ellipse
\[
E(c)=\left\{w:|a-w|+|\bar{a}-w|=c \right\}
\]
and require that (1) $E(c)\subset{\overline{\mathbb B}}^2$, (2) $E(c)\cap \partial{\B}\neq \emptyset$ and the $x$- coordinate of the point of contact of $E(c)$
and the unit circle is unique. Both requirements (1) and (2) can be met for a suitable choice of $c$. The major and minor semiaxes of the ellipse
are $c/2$ and $\sqrt{(c/2)^2-\beta^2}$, respectively. The point of contact can be obtained by solving the system
\[
 \left\{ \begin{array}{ll} x^2+y^2=1 \\
  \frac{(x-\alpha)^2}{(c/2)^2-\beta^2}+\frac{y^2}{(c/2)^2}=1.
   \end{array}\right.
\]
Solving this system yields a quadratic equation for $x$ with the discriminant
\[
D=64(c^2- 4 \beta^2)(\alpha^2 c^2+\beta^2 (c^2-4)).
\]
If the discriminant is positive, there are at least two points of intersection of the
unit circle and the ellipse. Because we are interested only in the case when there are
at most two points of tangency, we must require that $D=0\,.$ Because the length of the smaller semiaxis
$\sqrt{(c/2)^2-\beta^2}>0\,,$ we see that $D=0$ only if
\[
c=\frac{2\beta}{\sqrt{\alpha^2+\beta^2}}.
\]
In this case
\[
x=\frac{1}{32 \beta^2} 8 \alpha c^2=\frac{\alpha}{\alpha^2+ \beta^2}.
\]
The points $\{  w= x+ iy :  x= x^2 + y^2\}$ define the circle  $\left|w-1/2 \right|=1/2$
 and we have $\frac{\alpha}{\alpha^2+ \beta^2}>1$
if and only if $\left|a-1/2 \right|<1/2, a=\alpha+ i \beta \,. $
In the case $\frac{\alpha}{\alpha^2+ \beta^2}>1$ the contact point is $z=(1,0)$,
whereas in the case $\frac{\alpha}{\alpha^2+ \beta^2}<1$ the point is
\[
z=(x,\sqrt{1-x^2})=\left(\frac{\alpha}{\alpha^2+ \beta^2},\frac{\sqrt{(\alpha^2+ \beta^2)^2-\alpha^2}}{\alpha^2+ \beta^2}\right).
\]
We now compute the focal sum $c$ in both cases

\[
 \left\{ \begin{array}{ll} c=\frac{2 \beta}{\sqrt{\alpha^2+ \beta^2}}=\frac{2 \Im{a}}{|a|},\quad \text{if}\quad \left|a-1/2\right|\geq 1/2, \\
 \\
  c=2|a-(1,0)|=2\sqrt{\beta^2+(1-\alpha)^2},\quad \text{if}\quad \left|a-1/2\right|\leq 1/2.
   \end{array}\right.
\]
Finally we see that
\[
s_{\B}(a,\bar{a})=\frac{|a-\bar{a}|}{c}=|a|,\quad \text{if}\, \left|a-1/2\right|\geq 1/2,
\]
otherwise
\[
s_{\B}(a,\bar{a})=\frac{|a-\bar{a}|}{c}=\frac{\beta}{\sqrt{\beta^2+(1-\alpha)^2}}=\frac{\Im{a}}{\sqrt{(1-\Re{a})^2+(\Im{a})^2}}.\qedhere
\]
\end{proof}

\begin{thm}\label{sthm}
Let $x, y\in\B$ with $|x|=|y|$ and $z\in\partial\B$ such that $|y-z|<|x-z|$ and
\[
\measuredangle (y,z,0)=\measuredangle (0,z,x)=\gamma \,.
\]
Then $\cos{\gamma}=(|x-z|+|y-z|)/2$ and hence $|y-z|<1$. Moreover, $0, x, y, z$ are concyclic.
\end{thm}

\begin{figure}[h]
\begin{center}
     \includegraphics[width=7cm]{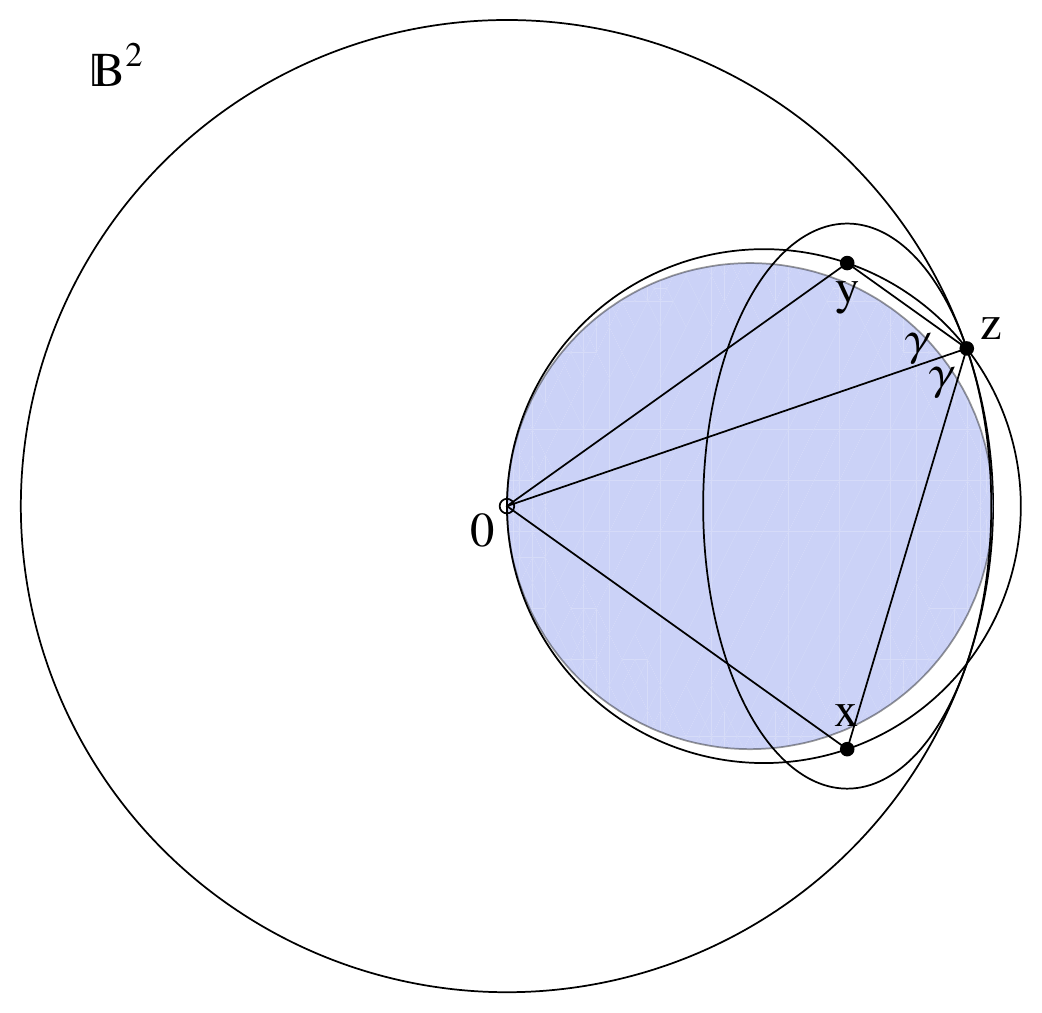}
\caption{ The ellipse with foci $x$ and $y$ internally tangent to the unit circle at $z, |x|=|y|\,.$ The points $0, x, z, y$ are concyclic by Theorem \ref{sthm}. }
\end{center}
    \end{figure}

\begin{proof}
By the Law of Cosines
\[
|x|^2=|x-z|^2+1-2|x-z|\cos{\gamma}
\]
and
\[
|y|^2=|y-z|^2+1-2|y-z|\cos{\gamma}.
\]
Because $|x|=|y|$ these equalities yield
\beq\label{cos1}
\cos{\gamma}=\frac{|x-z|+|y-z|}{2}.
\eeq
By Ptolemy's Theorem $0, x, z, y$ are concyclic if and only if
\[
|y-z||x|+|y||x-z|=1\cdot |x-y|,
\]
which is equivalent to
\beq\label{ptol}
|y-z|+|x-z|=\frac{|x-y|}{|x|}.
\eeq
By Theorem \ref{sb02}, we see that
\[
s_{\B}(x,y)=\frac{|x-y|}{|x-z|+|z-y|}=|x|,
\]
which proves \eqref{ptol}.
\end{proof}
%

\begin{cor}\label{sym}
Let $D\subset\B$ be a domain and let $x, -x\in D$. Then
\[
s_D(x,-x)\geq |x|.
\]
\end{cor}
\begin{proof}
It follows from Theorem \ref{sb02} that
\[
s_D(x,-x)\geq s_{\B}(x,-x)=|x|.\qedhere
\]
\end{proof}

\bigskip

%

\begin{thm}\label{zforsinBn}
  Let $x \in (0,1)$, $y \in \mathbb{B}^2 \setminus \{ 0 \}$, $\textnormal{Im} \, y \geq 0$, with $|y|=|x|$ and denote $\omega = \measuredangle (x,0,y)$. Then the supremum in \eqref{sm} is attained at $z = e^{i \theta}$ for
  \[
    \theta = \left\{ \begin{array}{ll} \frac{\omega}{2}, & \textrm{if } \sin \frac{\pi-\omega}{2} \geq |x|,\\ \frac{\omega-\pi}{2}+\arcsin \frac{\sin \frac{\pi-\omega}{2}}{|x|}, & \textrm{if } \sin \frac{\pi-\omega}{2} < |x|. \end{array} \right.
  \]
\end{thm}
\begin{proof}
  By \eqref{sm} and geometry it is clear that the supremum is attained at a point $z = e^{i \theta}$ with $\measuredangle (x,z,0) = \measuredangle (y,z,0)$. We denote this angle by $\gamma$. Since $\gamma = \measuredangle (x,z,0) = \measuredangle (y,z,0)$ and $|x|=|y|$ we obtain by the Law of Sines
  \beq\label{sines}
    \frac{1}{\sin(\pi-\theta-\gamma)} = \frac{|x|}{\sin \gamma} = \frac{1}{\sin(\pi-\omega+\theta-\gamma)},
  \eeq
  which is equivalent to
  \[
    \sin(\pi-\theta-\gamma) = \sin(\pi-\omega+\theta-\gamma).
  \]
  This has two solutions: $a=b$ or $a+b=\pi$, where $a=\pi-\theta-\gamma$ and $b=\pi-\omega+\theta-\gamma$. The solution $a=b$ gives
  \begin{equation}\label{solution1}
    \theta = \frac{\omega}{2}.
  \end{equation}
  The solution $a+b=\pi$ gives $\omega = \pi-2\gamma$. In this case by \eqref{sines} we obtain
 \begin{equation}\label{sine4}
    \frac{1}{\sin \left( \frac{\pi+\omega}{2}-\theta \right)} = \frac{|x|}{\sin \left( \frac{\pi-\omega}{2} \right)},
 \end{equation}
  which gives
  \begin{equation}\label{solution2}
    \theta = \frac{\pi+\omega}{2}-\arcsin \frac{\sin \frac{\pi-\omega}{2}}{|x|}.
  \end{equation}

  We have two solutions \eqref{solution1} and \eqref{solution2}. Next we find out which solution gives the supremum in \eqref{sm}. First we note that \eqref{solution2} is valid only for $\sin \frac{\pi-\omega}{2} \leq |x|$. Thus for $\sin \frac{\pi-\omega}{2} > |x|$ we choose \eqref{solution1}. In the case $\sin \frac{\pi-\omega}{2} = |x|$ both solutions give $\theta = \frac{\omega}{2}$. Thus, in the case $\sin \frac{\pi-\omega}{2} \geq |x|$, the supremum in \eqref{sm} is attained at $z=e^{i\omega/2}$.

  Finally, we consider the case $\sin \frac{\pi-\omega}{2} < |x|$. Let us denote $\theta_1=\frac{\omega}{2}$, $z_1=e^{i \theta_1}$, $ \theta_2 = \frac{\pi+\omega}{2}-\arcsin \frac{\sin \frac{\pi-\omega}{2}}{|x|}$ and $z_2=e^{i \theta_2}$. Moreover let $\omega_0=\measuredangle (0,x,z_1)$, $\omega_1=\measuredangle (0,x,z_2)$, $\omega_2=\measuredangle (0,y,z_2)$. Again by the Law of Sines, we obtain
  \beq\label{sin2}
  \frac{|x-z_2|}{\sin{\theta_2}}=\frac{1}{\sin{\omega_1}}=\frac{1}{\sin{\omega_2}}=\frac{|z_2-y|}{\sin(\omega-\theta_2)}:=k_1,
  \eeq
  and
  \beq\label{sin3}
  \frac{|x-z_1|}{\sin{\frac{\omega}{2}}}=\frac{1}{\sin{\omega_0}}:=k_2.
  \eeq
By \eqref{sine4}, we see that
$$
k_1=\frac{|x|}{\sin \left( \frac{\pi-\omega}{2} \right)}.
$$
By \eqref{sin2} and \eqref{sin3}, the inequality $|x-z_2|+|z_2-y| < |x-z_1|+|z_1-y|$ is equivalent to
  \[
    k_1(\sin \theta_2 + \sin (\omega-\theta_2)) < 2 k_2 \sin{\frac{\omega}{2}} .
  \]
  By substituting $k_1$ and $k_2$, it is enough to show that
  \[
  \frac{|x|}{\sin \left( \frac{\pi-\omega}{2} \right)} \cos\left(\theta_2-\frac{\omega}{2}\right) < \frac{1}{\sin{\omega_0}},
  \]
  which is, by substitution of $\theta_2$, equivalent to the inequality
  \[
  1 < \frac{1}{\sin{\omega_0}}.
  \]

Thus, in the case $\sin \frac{\pi-\omega}{2} < |x|$, the supremum in \eqref{sm} is attained at $z_2$.
\end{proof}

\begin{rem}
By the assumptions of Lemma \ref{zforsinBn}, if $\sin \frac{\pi-\omega}{2} \geq |x|$  we attain

\[
s_{\B}(x,y)=\frac{|x-y|}{|x-e^{i\omega/2}|+|y-e^{i\omega/2}|}=\frac{|x|\sin{\frac{\omega}{2}}}{\sqrt{1+|x|^2-2|x|\cos{\omega/2}}}.
\]
This formula is equivalent to \eqref{sbeta}, if $y=\bar{x}$, and thus by Theorem \ref{sb02} we collect
\[
  s_{\B}(x,y) = \left\{ \begin{array}{ll} |x|, & \cos(\omega/2) < |x|,\\ \displaystyle \frac{|x|\sin(\omega/2)}{\sqrt{1+|x|^2-2|x|\cos(\omega/2)}}, & \cos(\omega/2) \ge |x|, \end{array} \right.
\]
where $x,y \in \B$, $|y|=|x|$ and $\omega = \measuredangle(x,0,y)$.

Note that the following inequalities are equivalent:

$|a-\frac{1}{2}|\leq \frac{1}{2}$ where $a$ is as in Theorem \ref{sb02}, $\cos{\frac{\omega}{2}}\geq |x|$, $\frac{|x-y|}{2}\leq |x|\sqrt{1-|x|^2}$.
\end{rem}



\section{The proof of the main result}\label{section4}


\noindent{\bf Proof of Theorem \ref{2sc}.}
By a simple geometric observation we see that
\begin{equation}\label{cassinprod}
\inf_{w\in\partial \Bn}|x-w||w-y|\leq 1.
\end{equation}
In fact, for given $x,y\in\Bn$, let $x', y'\in\Bn$ be the points such that $y'-x'=y-x$ and $y'=-x'$. Then the size of the maximal Cassinian oval $C(x,y)$ with foci $x, y$ which is contained in the closed unit ball is not greater than that of the maximal Cassinian oval $C(x',y')$ with foci $x', y'$, see the Figure 2. 
\begin{figure}[h]\label{fig-cassini}
\begin{center}
     \includegraphics[width=7cm]{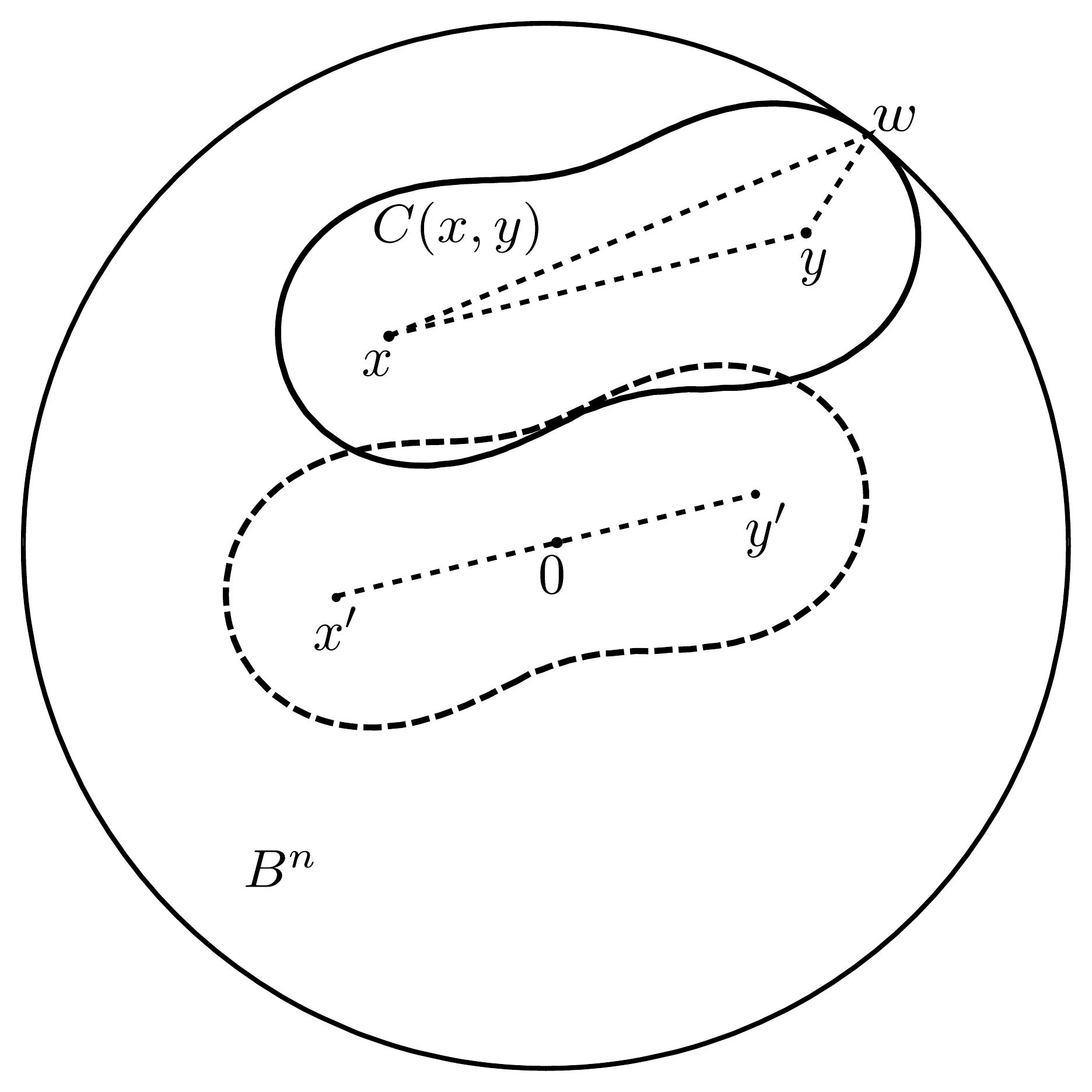}
\caption{The maximal Cassinian oval $C(x,y)$ is not larger than the maximal Cassinian oval $C(x',y')$.}
\end{center}
\end{figure}
This implies that
\allowdisplaybreaks\begin{align*}
\inf_{w\in\partial \Bn}|x-w||w-y| & \leq \inf_{w\in\partial \Bn}|x'-w||w-y'|\\
                                 & = 1-\left(\frac{|x-y|}{2}\right)^2 \leq 1.
\end{align*}

Therefore, for $x,y\in D\subset \Bn$, we have that
\begin{equation}\label{cassinprod4D}
\inf_{w\in\partial D}|x-w||w-y|\leq\inf_{w\in\partial \Bn}|x-w||w-y|\leq 1.
\end{equation}
For $x=y\in D$, the desired inequality is trivial.  For $x,y\in D$ with $x\neq y$,
it follows from the inequality of arithmetic and geometric means and the inequality \eqref{cassinprod4D} that
\allowdisplaybreaks\begin{align*}
\frac{c_{D}(x,y)}{2s_{D}(x,y)} &= \frac{\inf\limits_{w\in\partial D}(|x-w|+|w-y|)}{2\inf\limits_{w\in\partial D}(|x-w||w-y|)}\\
                                   &\geq \frac{\inf\limits_{w\in\partial D}\sqrt{|x-w||w-y|}}{\inf\limits_{w\in\partial D}(|x-w||w-y|)}\\
                                   &= \frac{\sqrt{\inf\limits_{w\in\partial D}(|x-w||w-y|)}}{\inf\limits_{w\in\partial D}(|x-w||w-y|)}\\
                                   &\geq 1.
\end{align*}
For the sharpness of the constant in the case of the unit ball, let $y=-x\to 0$. It is easy to see that both the inequality of arithmetic and geometric means and the inequality \eqref{cassinprod} will asymptotically become equalities. This completes the proof.  \hfill $\square$

\begin{cor}
Let $D\subset\Rn$ be a bounded domain. Then, for $x,y\in D$,
$$
c_D(x,y)\geq \frac{2}{\sqrt{n/(2n+2)} \operatorname{diam}(D)} s_D(x,y).
$$
\end{cor}
\begin{proof}
By the well-known Jung's theorem \cite[Theorem 11.5.8]{berger}, there exists a ball $B$ with radius $\sqrt{n/(2n+2)}\mbox{diam}(D)$ which contains the bounded domain $D$. Let $f$ be a similarity which maps the ball $B$ onto the unit ball $\Bn$. Then it is easy to see that for all $x,y\in B$,
$$
|f(x)-f(y)|=\frac{|x-y|}{\sqrt{n/(2n+2)}\mbox{diam}(D)}.
$$
By the definitions of the Cassinian metric and the triangle ratio metric, we have that for $x,y\in D$,
\begin{equation}\label{csdfd}
c_{fD}(f(x),f(y))=\sqrt{n/(2n+2)}\mbox{diam}(D)c_D(x,y)
\end{equation}
and
\begin{equation}\label{csdfd1}
 s_{fD}(f(x),f(y))=s_D(x,y).
\end{equation}
Since $fD\subset\Bn$, by Theorem \ref{2sc} we have
\begin{equation}\label{csfd}
c_{fD}(f(x),f(y))\geq 2 s_{fD}(f(x),f(y)).
\end{equation}
Combining \eqref{csdfd}, \eqref{csdfd1}, and \eqref{csfd}, we get the desired inequality.
\end{proof}

For some basic information about the Schwarz lemma the reader is referred to \cite{vu}.
In \cite{bv}, an explicit form of the Schwarz lemma for quasiregular mappings was given. In this Theorem we use the well-known distortion
function $ \varphi_K(r) $ of the Schwarz lemma, see \cite{vu,vw}. We also need the distortion function for
$K>1\,$ and  $0\leq t<\infty$
$$
\eta_K(t)=\frac{\varphi^2_K(\sqrt{t/(1+t)})}{1-\varphi^2_K(\sqrt{t/(1+t)})}\le e^{\pi(K-1/K)} \max \{t^{1/K}, t^K\}\,\,.
$$
See \cite[10.24, 10.35]{avv}.

\begin{thm}   \cite[Theorem 11.2,11.3]{vu}, \cite[Theorem 1.10]{bv}, \cite[Theorem 3.7]{vw}\label{bvtrans}
If $f:\B\rightarrow {\mathbb R}^2$ is a non-constant $K$-quasiregular map with $f\B\subset\B$ and $\rho$ is the hyperbolic metric of $\B$, then
$$\rho_{\B}(f(x),f(y))\le c(K) \max\{\rho_{\B}(x,y),\,\rho_{\B}(x,y)^{1/K}\}$$
for all $x,\,y\in\B$, where $c(K)=2\arth(\varphi_K({\rm th}\frac{1}{2}))\le 1.3507(K-1)+ K$ and, in particular, $c(1)=1$. Moreover,
if $f(0)=0\,,$ then for $x \in \B$
$$   |f(x)| \le \varphi_K(|x|)\le 4^{1-1/K} |x|^{1/K} \,.$$
\end{thm}

Combining Theorem \ref{bvtrans} with Theorems \ref{hvz216}- \ref{2sc} we obtain distortion results for quasiregular mappings of the
unit disk into unit disk with respect to the Cassinian metric.

\begin{thm}
If $f:\B\rightarrow {\mathbb R}^2$ is a non-constant $K$-quasiregular map with $f\B\subset\B$ and $f(0)=0$, then
$$c_{\B}(0,f(x))\le \eta_K(c_{\B}(0,x)) $$
for all $x\in\B$.
\end{thm}

\begin{proof}
By \cite[Example 3.9B]{i} and Theorem \ref{bvtrans}, we have
\begin{equation}\label{c0f}
c_{\B}(0,f(x))=\frac{|f(x)|}{1-|f(x)|}\leq\frac{\varphi_K(|x|)}{1-\varphi_K(|x|)}.
\end{equation}
It follows from \cite[Theorem 10.15]{avv} that
$$
\frac{\varphi_K(|x|)}{1-\varphi_K(|x|)}\leq\frac{\varphi^2_K(\sqrt{|x|})}{1-\varphi^2_K(\sqrt{|x|})}=\eta_K(|x|/(1-|x|))=\eta_K(c_{\B}(0,x)),
$$
which, combined with \eqref{c0f}, gives the desired result.
\end{proof}

The proof of Theorem \ref{casgrow} follows easily from the above results.

\bigskip
{\bf Acknowledgements.}  This research was started in International Conference on Geometric Function Theory and its Applications, December 18-21, 2014 in Kharagpur, India, organized by Professors B. Bhowmik and A. Vasudevarao.  P. Hariri and M. Vuorinen are indebted to Prof. S. Ponnusamy for making possible our participation to this meeting.
P. Hariri was supported by UTUGS, The Graduate School of the University of Turku.
X. Zhang was supported by the Academy of Finland project 268009.
The authors express their thanks to the referee for valuable suggestions.


\end{document}